\documentclass[oneside,english]{amsart}
\usepackage[T1]{fontenc}
\usepackage[latin9]{inputenc}
\usepackage{color}
\usepackage{textcomp}
\usepackage{amstext}
\usepackage{amsthm}
\usepackage{amssymb}
\usepackage{csquotes}
\usepackage[backend=bibtex8,maxbibnames=99,style=numeric]{biblatex}
\bibliography{CDM}

\AtEveryBibitem{\clearfield{month}
\clearfield{day}
\clearfield{doi}
\clearfield{url}
\clearfield{eprint}
\clearfield{isbn}
\clearfield{issn}
}

\renewbibmacro{in:}{}
\DeclareFieldFormat
  [article,inbook,incollection,inproceedings,patent,thesis,unpublished]
  {title}{#1\isdot}
\makeatletter
\def\url@leostyle{%
  \@ifundefined{selectfont}{\def\UrlFont{\sf}}{\def\UrlFont{\small\sffamily}}}
\makeatother
\makeatletter
\numberwithin{equation}{section}
\numberwithin{figure}{section}
\theoremstyle{plain}
\newtheorem{defi}{Definition}
\newtheorem{thm}{\protect\theoremname}
  \theoremstyle{plain}
  \newtheorem{cor}{\protect\corollaryname}
  \theoremstyle{plain}
  \newtheorem{lem}{\protect\lemmaname}
  \theoremstyle{remark}
  \newtheorem{rem}{\protect\remarkname}
    \newtheorem{obs}{Remark}
 \theoremstyle{plain}
 \newtheorem{lettertheorem}{Theorem}
  \theoremstyle{plain}
  \newtheorem{letterlemma}[lettertheorem]{Lemma}
  \theoremstyle{plain}

\DeclareMathOperator {\diam}{diam}
\DeclareMathOperator {\loc}{loc}

\newcommand{\R}{\mathbb{R}}
\usepackage{pgf,tikz}
\usetikzlibrary{arrows}
\usepackage{url}

\makeatother

\usepackage{babel}
  \providecommand{\lemmaname}{Lemma}
   \providecommand{\corollaryname}{Corollary}
  \providecommand{\remarkname}{Remark}
\providecommand{\theoremname}{Theorem}

\usepackage{hyperref}

\begin{document}

\title[Improved fractional Poincar\'e type inequalities on John domains]{Improved  fractional Poincar\'e type inequalities in John domains}

\author[M. E. Cejas]{Mar\'ia Eugenia Cejas}
\address[M. E. Cejas]{ Departmento de Matem\'atica, Universidad Nacional de La Plata, CONICET, La Plata, Argentina}
\email{ecejas@mate.unlp.edu.ar}

\author[I. Drelichman]{Irene Drelichman }
\address[I. Drelichman]{IMAS (UBA-CONICET), Facultad de Ciencias Exactas y Naturales, Universidad de Buenos Aires, Ciudad Universitaria, 1428 Buenos Aires, Argentina}
\email{irene@drelichman.com}

\author[J.C. Mart\'{i}nez-Perales]{Javier C. Mart\'{i}nez-Perales}
\address[J.C. Mart\'{i}nez-Perales]{
BCAM\textendash  Basque Center for Applied Mathematics, Bilbao, Spain}
\email{jmartinez@bcamath.org}

\begin{abstract}
	We obtain improved fractional Poincar\'e inequalities in John domains of a metric space $(X, d)$ endowed with a doubling measure $\mu$ under  some mild regularity conditions on the measure $\mu$. We also give sufficient conditions on a bounded domain to support  fractional Poincar\'e type inequalities in this setting.
\end{abstract}

\maketitle

\section{Introduction}

The aim of this paper is to study improved fractional Poincar\'e type inequalities in a bounded  John domain $\Omega \subset X$, where $(X,d,\mu)$ is a metric space endowed with a doubling measure $\mu$.

 Recall that, roughly speaking, $\Omega$ is a John domain if it has a central point such that any other point can be connected to it without getting too close to the boundary (see Section \ref{preliminares} for a precise definition), and that this is essentially the largest class of domains in $\R^n$ for which the Sobolev-Poincar\'e inequality with Lebesgue measure,
\begin{equation}\label{sobolev poincare}
\|u-u_{\Omega}\|_{L^\frac{np}{n-p}(\Omega)} \leq C \left(\int_{\Omega} |\nabla u(x)|^p \,dx \right)^{\frac{1}{p}},
\end{equation} 
 holds (see \cite{Ma,Re,Bo, Ha2} for the sufficiency, and \cite{BK95} for the necessity). Here,  $u$ is a locally Lipschitz function, $1\leq p < n$ and $u_\Omega$ is the average of $u$ over $\Omega$.

The above inequality, also called $(\frac{np}{n-p}, p)$-Poincar\'e inequality, is a special case of a larger family of so-called improved Poincar\'e inequalities, which are $(q,p)$-Poincar\'e inequalities with a weight that is a power of the distance to the boundary $d(x)$, namely,
$$
\|u-u_{\Omega}\|_{L^q(\Omega)} \leq C \|d^{\alpha}|\nabla u|\|_{L^p(\Omega)}
$$
where $1\le p\le q \le \frac{np}{n-p(1-\alpha)}$,  $p(1-\alpha)<n$ and $\alpha \in [0,1]$ (see \cite{BS,H1}, and also \cite{DD,ACD} for weighted versions).  The results above were based in the use of representation formulae in terms of a fractional integral, but we point out that a different approach to obtain Poincar\'e-Sobolev inequalities was introduced in \cite{FPW} (sharpened in \cite{MP}) which avoids completely any representation formula.  
See also the recent work \cite{PR} for more precise results.

It is also worth noting that these inequalities have also been studied in metric spaces with doubling measures, replacing $|\nabla u|$ by a generalized gradient (see  \cite{HK} and references therein).

In  recent years,  several authors have turned their attention to the fractional counterpart of inequality \eqref{sobolev poincare}, beginning with the work \cite{HV} where the inequality 
\begin{equation}\label{fraccionaria euclidea}
\|u-u_{\Omega}\|_{L^q(\Omega)} \leq C \left(\int_{\Omega}\int_{\Omega \cap B(x,\tau d(x))} \frac{|u(x)-u(z)|^p}{|x-z|^{n+sp}}\,dx\,dz\right)^{1/p} 
\end{equation}
was obtained for a bounded John domain $\Omega$, $s, \tau \in (0,1)$, $p<\frac{n}{s}$ and $1< p \leq q \leq \frac{np}{n-sp}$. The case $p=1$ was proved in \cite{DIV} using  the so-called Maz'ya truncation method (see \cite{Ma}) adapted to the fractional setting, which allows to obtain a strong inequality from a weak one. 

Observe that the seminorm appearing on the right hand side of inequality (\ref{fraccionaria euclidea}) is stronger than that of the usual fractional Sobolev space $W^{s,p}(\Omega)$. More precisely, if we consider $W^{s,p}(\Omega)$ to be the subspace of $L^p(\Omega)$ induced by the seminorm
$$
[f]_{W^{s,p}(\Omega)}:= \left(\int_{\Omega}\int_{\Omega} \frac{|u(x)-u(z)|^p}{|x-z|^{n+sp}}\,dx\,dz\right)^{1/p},
$$
and $\widetilde W^{s,p}(\Omega)$ to be the one induced by the seminorm
$$
[f]_{\widetilde W^{s,p}(\Omega)}:= \left(\int_{\Omega}\int_{\Omega \cap B(z,\tau d(z))} \frac{|u(x)-u(z)|^p}{|x-z|^{n+sp}}\,dx\,dz\right)^{1/p},
$$
 for fixed $s,\tau \in (0,1)$, then it is known that both spaces coincide when $\Omega$ is Lipschitz (see \cite{D06}), but there are examples of John domains $\Omega \subset \R^n$ for which the inclusion $W^{s,p}(\Omega)\subset \widetilde W^{s,p}(\Omega)$ is strict (see \cite{DD3} for this result and characterizations of both spaces as interpolation spaces). This has led to call inequality \eqref{fraccionaria euclidea} an ``improved'' fractional inequality. However, throughout this work, we will use this terminology to refer to inequalities including powers of the distance to the boundary as weights, as in the classical case. 
 
Improvements of inequality \eqref{fraccionaria euclidea} in this sense were obtained in \cite{DD2} by including powers of the distance to the boundary to appropriate powers as weights on both sides of the inequality, and in \cite{LH}, where the weights are defined by powers of the distance to a compact set of the boundary of the domain.

In this article we deal with the natural problem of extending the fractional inequalities mentioned above to metric measure spaces. To the best of our knowledge the results about fractional Poincar\'e inequalities in this setting are new. 
We will consider a metric measure space $(X,d,\mu)$, where $\mu$ is a Borel doubling measure satisfying some mild regularity assumptions.
 To this end, for a given  $\Omega\subset X$, $1\le p<\infty$, $\tau, s\in (0,1)$, and a weight $w$ (i.e., a locally integrable $\mu$-almost everywhere positive function), we define the seminorm
\[[u]_{W_{\tau}^{s,p}(\Omega,w)}:=\left(\int_\Omega\int_{\Omega\cap \{d(z,y)\leq\tau d(y)\}}\frac{|u(z)-u(y)|^pw(z,y)d\mu(z)d\mu(y)}{\mu[B(z,d(z,y))]d(z,y)^{s p}}\right)^{1/p}.
\]

It should be pointed out that this is the natural extension of the fractional Sobolev space to the context of metric spaces. Indeed, according to \cite{GKS}   the Besov space $B_{p,q}^s(X,d\mu)$ defined by the norm
$$\left(\int_X \int_X \frac{|u(x)-u(y)|^p}{\mu[B(x,d(x,y))]d(x,y)^{s p}}\,d\mu(x)\,d\mu(y) \right)^{1/p}$$ 
 coincides with the interpolation space between  $L^p(X,d\mu)$ and the Sobolev space $W^{1,p}(X,d\mu)$ in the case that the metric supports a Poincar\'e inequality. More details about this fact and Besov spaces can be found in \cite{GKS,HK,Ha,PP,Ya}.

In this work, we are interested in the study of inequalities of the form
\begin{equation}\label{Def.Poincare}
\inf_{a\in \mathbb{R}}\|u-a\|_{L^{q}(\Omega,wd\mu)}\lesssim [u]_{W_{\tau}^{s,p}\left(\Omega,vd\mu \right)},
\end{equation}	
where  $1\leq p,q< \infty$, $s,\tau\in(0,1)$ and $w,v$ are weights. We will say that $\Omega$ supports the $(w,v)$-weighted fractional $(q,p)$-Poincar\'e inequality in $\Omega$ if \eqref{Def.Poincare} holds on $\Omega$ for every function $u\in L^p(\Omega,wd\mu)$ for which the right hand side is finite. When $w$ and $v$ are defined by functions of the distance to the boundary, we shall refer to these inequalities as $(w,v)$-improved fractional inequalities.

Our first goal is to obtain  such inequalities with weights  of the form $w_{\phi }^F(x)=\phi(d_F(x))$, where $\phi$ is a positive increasing function satisfying a certain growth condition and $F$ is a compact set in $\partial\Omega$. The parameter $F$ in the notation will be omitted whenever $F=\partial\Omega$. At the right hand side of the inequality, we will obtain a weight of the form $v_{\Phi,\gamma }^F(x,y)=\min_{z\in\{x,y\}}d(z)^{\gamma} \Phi(d_F(z))$, where $\Phi$ will be an appropriate power of $\phi$.  Our results extend and improve results in \cite{HV, LH, DD2} in several ways. On one hand, the obtained class of weights is larger than the ones previously considered, even in the Euclidean setting. On the other hand, our inequalities hold for a very general class of spaces. Among these we can find, of course, the usual Euclidean space, but other important examples are included (see the discussion in Section \ref{ejemplos}).

Our second goal is to prove a sufficient condition for a domain $\Omega$ and a function $\phi$ to support an improved $(q,p)$-Poincar\'e inequality considering weights $w_\phi$ of the same type as the ones obtained above. In the Euclidean case it is well-known that if $q<p$, inequality \eqref{fraccionaria euclidea} does not hold for general domains. Indeed, it was shown in \cite[Theorem 6.9]{HV} that there exists a $\delta$-John domain  which does not support this inequality. 
Following the ideas in \cite{HV} we obtain geometric sufficient conditions on a bounded domain $\Omega \subset X$ and a function $\phi$ to support an improved fractional $(q,p)$-Poincar\'e-Sobolev inequality when $q\leq p$ in the setting of metric measure spaces.

The rest of the paper is organized as follows: in Section \ref{preliminares} we introduce some necessary notations and previous results; Section \ref{demostraciones} is devoted to prove our $(w,v)$-improved fractional $(q,p)$-Poincar\'e inequalities for $1\le p\le q<\infty$. In Section \ref{ejemplos} we consider some special cases of those inequalities. In Section \ref{suficiencia del dominio} we give the proof of the sufficient condition for a bounded domain to support the $(q,p)$-Poincar\'e inequality for $q \le p$.   
Finally, in Section \ref{ejemplo suficiencia} we give an example of a domain satisfying the condition of Theorem \ref{sufi}.

\section{Notations and Preliminaries}\label{preliminares}

From now on $C$ and $c$ will denote constants that can change their values even within a single string of estimates. When necessary, we will stress the dependence of a constant on a particular parameter by writing it as a subindex. Also, we will use the notation $A\lesssim B$ whenever there exists a  constant $c>0$ independent of all relevant parameters for which $A\leq cB$.  Whenever $A\lesssim B$ and $B\lesssim A$, we will write $A\asymp B$.

A metric space $(X,d)$ is a set $X$ with a metric $d$, namely a nonnegative function defined on $X \times X$ such that
\begin{enumerate}
\item For every $(x,y) \in X \times X$, $d(x,y)=0$ if and only if $x=y$.
\item For every $(x,y) \in X \times X$, $d(x,y)=d(y,x)$.
\item The inequality $d(x,y) \leq d(x,z)+ d(y,z)$ holds for every $x,y,z \in X$.
\end{enumerate}

 The distance between a point $x$ and a subset $F$ of the boundary of $\Omega$ will be denoted $d_F(x):=d(x,F)$. When $F=\partial\Omega$, we will simply write $d(x):=d(x,\partial\Omega)$. For given $r>0$ and $x\in X$, the ball centered at $x$ with radius $r$ is the set $B(x,r):=\{y\in X:d(x,y)<r\}$. Given a ball $B\subset X$,  $r(B)$ will denote its radius and $x_B$ its center. For any $\lambda>0$,  $\lambda B$ will be the ball with same center as $B$ and radius $\lambda r(B)$.

A doubling metric space is a metric space $(X,d)$  with the following (geometric) doubling property:
there exists a positive integer $N\in \mathbb{N}$ such that, for every $x\in X$ and $r>0$, the ball $B(x,r)$ can be covered by at most $N$ balls $B(x_i,r/2)$ with $x_1, \dots, x_N \in X$. 

Every doubling metric space $(X,d)$ has a dyadic structure, as was proved by  Hyt\"onen and  Kairema in the following theorem \cite{HyK}:
\begin{lettertheorem}\label{dyadiccubes}
Suppose that there are constants $0<c_0\leq C_0<\infty$ and $s\in (0,1)$ such that
\[12C_0s\leq c_0.\]
Given a set of points $\{z_{j}^k\}_{j}$, $j\in\mathbb{N}$, for every $k\in\mathbb{Z}$ -- called dyadic points -- with the properties that
\begin{equation}\label{dyadic1}
d(z_i^k,z_j^k)\geq c_0s^k, \quad i\neq j,\qquad \min_{j\in\mathbb{N}}d(x,z_j^k)<c_0s^k,\quad x\in X,
\end{equation}
we can construct families of Borel sets $\tilde{Q}_j^k\subset Q_j^k\subset \overline{Q}_j^k$ -- called open, half open and closed dyadic cubes -- such that:
\begin{equation}\label{dyadic2}
\tilde{Q}_j^k\text{ and }\overline{Q}_j^k\text{ are interior and closure of }Q_j^k, \text{ respectively};
\end{equation}
\begin{equation}\label{dyadic3}
\text{if }l\geq k\text{, then either }Q_i^l\subset Q_j^k\text{ or }Q_j^k\cap Q_i^l=\emptyset;
\end{equation}
\begin{equation}\label{dyadic4}
X=\mathring{\bigcup}_{j\in\mathbb{N}}Q_j^k,\qquad k\in\mathbb{Z};
\end{equation}
\begin{equation}\label{dyadic5}
b(Q_j^k):=B(z_j^k,c_1s^k)\subset \tilde{Q}_j^k\subset Q_j^k\subset \overline{Q}_j^k\subset B(z_j^k,C_1s^k)=:B(Q_j^k),
\end{equation}
where $c_1:=\frac{c_0}{3}$ and $C_1:=2C_0$;
\begin{equation}\label{dyadic6}
\text{if }k\leq l\text{ and }Q_i^l\subset Q_j^k\text{, then }B(Q_i^l)\subset B(Q_j^k).
\end{equation}
\end{lettertheorem}

\begin{rem} \
\begin{enumerate}
\item One can always find a family of points $\{z_j^k\}_{j\in\mathbb{N},k\in\mathbb{Z}}$ as the one in the hypothesis, so a doubling metric space always has a dyadic structure. 
\item The open and closed cubes are indeed open and closed sets, respectively. 
\end{enumerate}
\end{rem}

The metric space $(X,d)$ will be endowed with a (always nonzero) Borel measure $\mu$ and denoted by $(X, d, \mu)$.
 If $E$ is a measurable set and $u$ is a measurable function, we write $u_E=\frac{1}{\mu(E)}\int_Eud\mu$ for the average of $u$ over $E$. We say that $\mu$ is doubling if for any ball $B \subset X$ there exists a constant $C_d$ depending on the measure such that $\mu(2B) \leq C_d\mu(B)$. Equivalently, $\mu$ is a doubling (or $n_\mu$-doubling) measure if there exist constants $c_d,n_\mu>0$ such that
\begin{equation}\label{doubling}
\frac{\mu(\tilde{B})}{\mu(B)}\leq c_d\left(\frac{r(\tilde{B})}{r(B)}\right)^{n_\mu},\end{equation}
for any pair of balls $B\subset \tilde{B}$ in $X$. It turns out that any metric space endowed with a doubling measure is a doubling metric space. Observe that, if $0\leq n_\mu\leq \eta$, then \eqref{doubling} also holds for $\eta$ instead of $n_\mu$.

We will also need to consider another property for the doubling measure $\mu$. Take $\delta>0$. We say that a measure $\mu$ is a $\delta$-reverse doubling measure, if there exists $c_r>0$ such that
\begin{equation}\label{reverse_doubling}
c_\delta\left(\frac{r(\tilde{B})}{r(B)}\right)^\delta\leq \frac{\mu(\tilde{B})}{\mu(B)},
\end{equation}
for any pair of balls $B\subset \tilde{B}$ in $X$. 
Observe that, if $0\leq s\leq \delta$, then \eqref{reverse_doubling} also holds for $s$ instead of $\delta$. 

It should be noted that, for an $n_\mu$-doubling and $\delta$-reverse doubling measure $\mu$, the relation $\delta\leq n_\mu$ must be satisfied. When we say that $\mu$ is $\delta$-reverse doubling, we will always assume that $\delta$ is the biggest exponent for which there exists $c_\delta>0$ such that \eqref{reverse_doubling} holds. Analogously, $n_\mu$ will be assumed to be the smallest exponent for which there exists $c_d$ such that \eqref{doubling} holds.

The reverse doubling property is not too restrictive, as doubling measures are reverse doubling whenever the metric space satisfies some metric property called uniform perfectness (see, e.g., \cite{W,TRTOVS}). Also, it is known (see \cite[p.112]{JH}) that Ahlfors-David regular spaces are precisely (up to some transformations) those uniformly perfect metric spaces carrying a doubling measure.

 Recall that, for $\eta>0$, a measure $\mu$ is $\eta$-Ahlfors-David regular if there exist constants $c_l,c_u>0$ such that
 \begin{equation}\label{ahlfors}
 c_lr(B)^\eta\leq \mu(B)\leq c_u r(B)^\eta,
 \end{equation}
 for any ball $B$ with $x_B\in X,\ 0<r(B)<\diam X$.
 The measure $\mu$ will be called $\eta$-lower Ahlfors-David regular if it satisfies just the left-hand side inequality and will be called $\eta$-upper Ahlfors-David regular if it satisfies just the right-hand side one.
 
 For a subset $E\subset X$, we will say that $\mu$ is (resp. lower or upper) Ahlfors-David regular on $E$ if  the induced subspace $(E,d|_E,\mu|_E)$ is (resp. lower or upper) Ahlfors-David regular.
 
Along this paper, $\Omega$ will be a domain, i.e. an open connected set. The well-known Whitney decomposition in the Euclidean case can be extended to a doubling metric space $(X,d)$, see for instance \cite{CW,MS,HKST}.

\begin{letterlemma}[Whitney decomposition \cite{HKST}]\label{Whitney}Let $\Omega$ be domain of finite measure strictly contained in $X$. For fixed $M>5$, we can build a covering $W_M$ of $\Omega$ by a countable family of balls $\{B_i\}_{i\in \mathbb{N}}$ such that $B_i=B(x_i,r_i)$, with $x_i\in \Omega$ and $r_i=r(B_i)=\frac{1}{M}d(x_i)$, $i\in\mathbb{N}$.
\end{letterlemma}

\begin{obs}\label{dilatadas y cadenas}
The Whitney decomposition can be taken such that, if one denotes $B^*$ the dilation of a ball $B$, then, when the chosen dilation is sufficiently small compared to $M$, we get the following properties for any ball $B\in W_M$ and their dilations:
\begin{enumerate}
\item $B^{*}\subset \Omega$;
\item $c_M^{-1}r(B^{*})\leq d(x)\leq c_Mr(B^{*})$, for all $x\in B^{*}$;
\item $\sum_{B\in W_M} \chi_{B^*}\lesssim \chi_\Omega$.
\end{enumerate}
We note that, for a fixed ball $B_0\in W_M$ and any ball $B$ in $W_M$, it is possible to build a finite chain $\mathcal{C}(B^*):=(B_0^*,B_1^*,\ldots,B_k^*)$,  with $B_k=B$ in such a way that $\lambda B_i\cap \lambda B_{j}\neq \emptyset$ for  some $\lambda>1$ sufficiently small (this $\lambda$ has something to do with the chosen dilation in the definition of $B^*$)  if and only if $|i-j|\leq1$.  With this definition of $B^*$, we have that, for two consecutive balls $B_j^*$ and $B_{j+1}^{*}$ in a chain like this, the following property holds
\begin{equation}\label{chainprop}
\max\{\mu(B_j^*),\mu(B_{j+1}^*)\}
\leq c_{M,\lambda} \mu\left( B_j^*\cap B_{j+1}^*\right).
\end{equation}

We denote $\ell[\mathcal{C}(B^*)]$ the length $k$ of this chain. Once the chains have been built, we can define, for each Whitney ball $E\in W_M$, the shadow of $E$ as the set $E(W_M)=\{B\in W_M: E^*\in \mathcal{C}(B^*)\}$. This construction is called a chain decomposition of $\Omega$ with respect to the fixed ball $B_0$.
\end{obs}

The type of domains we consider in this work are the so-called John domains, first appeared in  \cite{Jo}, and systematically studied since the work \cite{MarS}.
\begin{defi}  A domain $\Omega \subset X$ is a John domain if there is a distinguished point $x_0 \in \Omega$ called central point and a positive constant $c_J$ such that every point $x \in \Omega$ can be joined to $x_0$ by a rectifiable curve $\gamma:[0,l]\rightarrow \Omega$ parametrized by its arc length for which $\gamma(0)=x$, $\gamma(l)=x_0$ and
$$d(\gamma(t),\partial \Omega)\geq \frac{t}{c_J} \hspace{0,3cm} \mbox{ for } t\in [0,l]. $$
\end{defi}

The following chaining result for John domains in doubling metric spaces is a slightly modified version of the chaining result in \cite[Theorem 9.3]{HK}: 

\begin{lettertheorem}\label{chain condition}
Let $\Omega$  be a John domain in a doubling metric space $(X,d)$. Let $C_2\geq 1$ and $x_0$ the central point of $\Omega$. Then, there exist a ball $B_0$ centered at $x_0$, and a constant $c_2$ that depends on the John constant of $\Omega$, the doubling constant of the space and $C_2$ such that for every $x \in \Omega$ there is a chain of balls $B_i=B(x_i,r_i)\subset \Omega$, $i=0,1,....,$ with the following properties
\begin{enumerate}
\item There exists $R_i\subset B_i\cap B_{i+1}$ such that $B_i\cup B_{i+1} \subset c_2R_i$, $i=0,1,\ldots$;
 \item $d(x,y) \leq c_2 r_i$, for any $y\in B_i$, $i=0,1,\ldots$, with $r_i \rightarrow 0$ for $i \rightarrow \infty$;
  \item $d(B_i,\partial\Omega)\geq C_2 r_i$, $i=0,1,\ldots$;
 \item $\sum_j \chi_{B_j} \leq c_2 \chi_{\Omega}$.
\end{enumerate} 
\end{lettertheorem}
\begin{obs} The sequence of balls obtained in \cite[Theorem 9.3]{HK} is finite, but it can be easily completed to a (possibly) infinite one with the properties mentioned above (see also \cite[Lemma 4.9]{HV} for the case of $\mathbb{R}^n$).  \end{obs} 
 
 \begin{obs}
 Theorem \ref{chain condition} holds for a larger class of domains called weak John domains, i.e., those domains $\Omega$ for which there are a central point $x_0$ and a constant $c_J\geq 1$ such that for every $x \in \Omega$ there exists a curve $\gamma :[0,1]\rightarrow \Omega$ with $\gamma(0)=x$, $\gamma(1)=x_0$ and
 $$d(\gamma(t),\Omega^c)\geq  \frac{d(x,\gamma(t))}{c_J},\qquad t\in [0,1].$$
 In fact, \cite[Theorem 9.3]{HK} (and our modified version) is proved for bounded weak John domains. However, in this paper we will restrict ourselves to usual John domains.

 \end{obs}

We will also need the following result about the boundedness of certain type of operators from $L^p(X,\mu)$ to $L^q(X,\mu)$. This result can be found in a more general version in \cite[Theorem 3]{SW}.

\begin{lettertheorem}\label{sawyer y wheeden} Let $(X,d, \mu)$ be a metric space endowed with a doubling measure $\mu$. Set $1<p<q<\infty$.  Let $T$ be an operator given by
$$Tf(x)=\int_X K(x,y) f(y)\,d\mu(y),\qquad x \in X, $$
where $K(x,y)$ is a nonnegative kernel.

Let us define
\begin{equation}\label{definicion de phi} \varphi(B)=\sup\lbrace K(x,y): x,y \in B, d(x,y) \geq  C(K) r(B)\rbrace,
\end{equation}
where $B$ is a ball of radius $r(B)$ and $C(K)$ is a sufficiently small positive constant that depends only on the metric $d$ and the kernel $K$.

Suppose that 
\begin{equation}\label{condicion del peso SW}
\sup_{B\subset X}\varphi(B) \mu(B)^{\frac{1}{q}+\frac{1}{p'}}<\infty,
\end{equation}
where the supremum is taken over all the balls $B \subset X$. 

Under these hypotheses,
\begin{equation} \label{acotacion sawyer y wheeden}\left( \int_X |Tf(x)|^q \,d\mu(x)\right)^{1/q} \lesssim \left(\int_ X f(x)^p \,d\mu(x) \right)^{1/p} . 
\end{equation}
\end{lettertheorem}

This result can be used to bound the fractional integral operator in our context under mild conditions on the measure. Recall that

\begin{equation}\label{fractional integral}
I^\mu_{\alpha}f(x):= \int_{X } \frac{f(y)d(x,y)^{\alpha}}{\mu[B(x,d(x,y))]}d\mu(y),\qquad 0< \alpha.
\end{equation}
By $I_0^\mu$ we understand the Hardy-Littlewood maximal function associated to the measure $\mu$, namely 

\[I_0^\mu f(x)= M_\mu f(x):=\sup_{B\ni x} \frac{1}{\mu(B)}\int_B |f(x)|d\mu(x).\]

When $\alpha=0$ (i.e., for the maximal function), the only possibility is $p=q$ and, in this case, it is enough to ask the measure $\mu$ to be doubling. Let then $\alpha>0$ and $p>1$. For some $q> p$ to be chosen later, we shall bound this operator from $L_\mu^{q'}$ to $L_\mu^{p'}$ using Theorem \ref{sawyer y wheeden}. 

To check that \eqref{condicion del peso SW} holds, fix a ball $B$. Then, by the doubling condition, for any $x,y \in B$ with $d(x,y)\geq C_{\mu,\alpha}r(B)$ we  have that
$$\frac{d(x,y)^{\alpha}}{\mu[B(x,d(x,y))]}\mu(B)^{1/p'+1/q}\asymp \frac{r(B)^{\alpha}}{\mu(B)} \mu(B)^{1/p'+1/q},$$
and then a sufficient condition for the boundedness of our operator is
$$\sup_{B\subset X} r(B)^{\alpha}\mu(B)^{1/q-1/p}<\infty.$$
This tells us that if our measure $\mu$ is $\alpha\frac{pq}{q-p}$-lower Ahlfors-David regular,  then the claimed boundedness holds.

Observe that, as can be deduced from \cite[Theorem 3]{SW}, if the measure is $\alpha$-reverse doubling, then the $\alpha\frac{pq}{q-p}$-lower Ahlfors-David regularity is a necessary and sufficient condition for the boundedness of the operator.

If we let $\eta:= \alpha\frac{pq}{q- p}$, then we may write $q$ in the form $q=\frac{\eta p}{\eta-\alpha p}$. It is immediate that measures that are Ahlfors-David regular on the whole space are automatically doubling and reverse doubling.

\section{Fractional Poincar\'e-Sobolev inequalities on John domains}\label{demostraciones}
This section is devoted to the study of improved fractional $(q,p)$-Poincar\'e inequalities on bounded John domains,  where  $1\leq p\leq q<\infty$. Particular cases of these inequalities include  
some already known results in the Euclidean case, such as the unweighted inequalities considered in  \cite{HV} and the inequalities where the weights are  powers of the distance to the boundary considered in  \cite{DD2, LH}. We will come back to these special examples in Section \ref{ejemplos}.

The proof makes use of some of the arguments in \cite{HV} and \cite{DD2}. The fundamental idea is the classical one to obtain ordinary Poincar\'e inequalities: to bound the oscillation of the function $u$ by the fractional integral of its derivative by using regularity properties of the function and the domain and then use the boundedness properties of the fractional integral operator. Thus, if we understand the function
\begin{equation}\label{fractional derivative}
g_p(y):=\left[\int_{\{z\in\Omega:d(z,y)\leq\tau d(y)\}}\frac{|u(y)-u(z)|^p}{\mu[B(z,d(y,z))]d(y,z)^{s p}}d\mu(z)\right]^{1/p}\chi_\Omega(y),\qquad y\in \Omega
\end{equation}
as a fractional version of the derivative of $u$ on $\Omega$,  we just have to bound $|u(x)-a|$ for some $a\in \mathbb{R}$ by its fractional integral, $I_s^\mu g_p(x)$, for $\mu$-a.e. $x\in \Omega$ (see \eqref{fractional integral} for the definition of the fractional integral). This is done in the following lemma.

\begin{lem}\label{estimacion puntual} Consider a bounded John domain $\Omega$ in the doubling metric space $(X,d,\mu)$ with a chain as the one in Theorem \ref{chain condition}. Suppose $\mu$ to be $\delta$-reverse doubling. Let $s,\tau\in (0,1)$, $0\leq s\leq \delta$  and $1\leq p<\infty$.  There exists $c_3>0$ such that, for any  $u \in W^{s,p}(\Omega,\mu)$. 
\[
|u(x)-u_{B_0}|\lesssim I_s^\mu\left(g_p\chi_{\Omega\cap\{d(\cdot,x)\leq c_3d(\cdot)\}}\right),\qquad \mu-\text{a.e. }x\in\Omega,
\]
where $g_p$ is as in \eqref{fractional derivative}. 
\end{lem}
\begin{proof}
We can use the chain of balls in Theorem \ref{chain condition} and the Lebesgue differentiation theorem in order to obtain, for $\mu$-almost every $x\in \Omega$, 
$$u(x)=\lim_{i \rightarrow \infty} \frac{1}{\mu(B_i)}\int_{B_i} u(y)d\mu(y)= \lim_{i \rightarrow \infty} u_{B_i}.$$
Fix one of these points $x\in \Omega$. Then, as consecutive balls in the chain intersect in a ball whose dilation contain the union of them, we have
\[
\begin{split}
|u(x)-u_{B_0}| &\leq
\sum_{i=0}^\infty|u_{B_{i+1}}-u_{B_i}|\\
&\leq \sum_{i=0}^\infty|u_{B_{i+1}}-u_{B_i\cap B_{i+1}}|+|u_{B_{i}}-u_{B_i\cap B_{i+1}}|\\
&\lesssim\sum_{i=0}^\infty \frac{1}{\mu(B_i)}\int_{B_i}|u(y)-u_{B_i}|d\mu(y).
\end{split}
\]

Now, observe that, for $z,y \in B_i$ we have that $B(y,d(z,y))\subset 3 B_i$. Then, we have that each term in the sum above can be bounded as follows
\[
\begin{split}
\frac{1}{\mu(B_i)}\int_{B_i}|u(y)-u_{B_{i}}|d\mu(y)&\leq\frac{1}{\mu(B_i)}\int_{B_i}\left|\frac{1}{\mu(B_i)}\int_{B_i}(u(y)-u(z))d\mu(z)\right|d\mu(y)\\
&\lesssim \frac{r_i^s}{\mu(B_i)}\int_{B_i}\left(\int_{B_i}\frac{|u(y)-u(z)|^pd\mu(z)}{\mu(B_i)d(y,z)^{  s p}}\right)^{1/p}d\mu(y)\\
&\lesssim \frac{r_i^s}{\mu(B_i)}\int_{B_i}\left(\int_{B_i}\frac{|u(y)-u(z)|^pd\mu(z)}{\mu[B(y,d(y,z))]d(y,z)^{s p}}\right)^{1/p}d\mu(y).
\end{split}
\]

According to condition (3) from Theorem \ref{chain condition} we have $d(B_i,\partial\Omega)\geq C_2r_i$ for every $i=0,1,\ldots$, so 
\[d(y)\geq C_2r_i,\qquad y\in B_i,\]
and thus, for any $y,z\in B_i$ we can write $d(y,z)\leq 2r_i\leq \frac{2}{C_2}d(y)$. Hence, by choosing $C_2=\frac{2}{\tau}$,
\[
\begin{split}
&\frac{1}{\mu(B_i)}\int_{B_i}|u(y)-u_{B_i}|d\mu(y)\\
&\lesssim \frac{r_i^s}{\mu(B_i)}\int_{B_i}\left(\int_{\{z\in\Omega:d(y,z)\leq \tau d(y)\}}\frac{|u(y)-u(z)|^pd\mu(z)}{\mu[B(y,d(y,z))]d(y,z)^{s p}}\right)^{1/p}d\mu(y)\\
&= \frac{r_i^s}{\mu(B_i)}\int_{B_i}g_p(y)d\mu(y).
\end{split}
\]

Therefore
$$\sum_{i=0}^{\infty} \frac{1}{\mu(B_i)} \int_{B_i} |u(y)-u_{B_i}|d\mu(y) \lesssim  \sum_{i=0}^{\infty} \frac{r_i^s}{\mu(B_i)} \int_{B_i} g_p(y) d\mu(y).$$

Now, by (2) in Theorem \ref{chain condition}, the doubling and $\delta$- reverse doubling property of $\mu$, we have that
\[
\frac{r_i^s}{\mu(B_i)} \int_{B_i} g_p(y) d\mu(y)\lesssim \int_{B_i} \frac{g_p(y)d(x,y)^s}{\mu[B(x,d(x,y))]} d\mu(y),\qquad i=0,1,\ldots,
\]
and from this and the fact that $d(x,y)\leq c_2/C_2 d(y)$ for every $y\in B_i$, we deduce
\[
\sum_{i=1}^\infty \frac{1}{\mu(B_i)} \int_{B_i} |u(y)-u_{B_i}|d\mu(y) \lesssim  \int_{\Omega\cap\{d(x,y)\leq c_3d(y)\}}  \frac{g_p(y)d(x,y)^s}{\mu[B(x,d(x,y))]} d\mu(y),
\]
where $c_3:=c_2/C_2=\frac{\tau c_2}{2}$.

\end{proof}

We will also need the following lemma for our main theorem. 
\begin{lem}\label{lema maximal}
Let $s>0$. Then, for any $x\in X$ and any $\varepsilon>0$ we have
\[
I_s^\mu(f\chi_{\Omega\cap\{d(x,\cdot)<\varepsilon\}})(x)
\lesssim\varepsilon^s I_0^\mu(f\chi_\Omega)(x).
\]
\end{lem}
\begin{proof}
A standard argument (see \cite{Hed}) dividing the integral gives us, for any $x\in X$,
\[
\begin{split}
I_s^\mu(f&\chi_{\Omega\cap\{d(x,\cdot)<\varepsilon\}})(x)=\int_{d(x,y)<\varepsilon}\frac{f(y)\chi_\Omega(y)}{\mu[B(x,d(x,y))]d(x,y)^{-s}}d\mu(y)\\
& =\sum_{k=0}^\infty\int_{B\left(x,\frac{\varepsilon}{2^k}\right)\backslash B\left(x,\frac{\varepsilon}{2^{k+1}}\right)}\frac{f(y)\chi_\Omega(y)}{\mu[B(x,d(x,y))]d(x,y)^{-s}}d\mu(y)\\
&\lesssim \varepsilon^s\sum_{k=0}^\infty \frac{1}{\mu\left[B\left(x,\frac{\varepsilon}{2^{k}}\right)\right]2^{s(k+1)}}\int_{B\left(x,\frac{\varepsilon}{2^k}\right)\backslash B\left(x,\frac{\varepsilon}{2^{k+1}}\right)}f(y)\chi_\Omega(y)d\mu(y)\\
&\lesssim \varepsilon^s I_0^\mu(f\chi_\Omega)(x).
\end{split}
\]
\end{proof}

Now we are ready to prove the main results of the section. Recall that we will work with weights of the form $w_{\phi}^F(x)=\phi(d_F(x))$, where $F$ is omitted whenever $F=\partial\Omega$, and that $\phi$ is a positive increasing function that satisfies the growth condition $\phi(2x)\le C \phi(x)$ for all $x\in\R_+$. Observe that this implies $\phi(kx)\le C_k \phi(x)$ for every $k\ge 1$.  We will obtain, at the right hand side of the inequality, a weight of the form $v_{\Phi,\gamma p}^F(x,y)=\min_{z\in\{x,y\}}d(z)^{\gamma p} \Phi(d_F(z))$, where $\Phi$ is an appropriate power of $\phi$.

\begin{thm}\label{P-S inequality}
	Let  $(X,d, \mu)$ be a metric space with $n_\mu$-Ahlfors-David regularity.  Let  $s,\tau\in(0,1)$ and $0\leq \gamma<s\leq n_\mu$. Let  $1<p<\infty$ be such that $(s-\gamma)p<n_\mu$ and take $ p^*_{s-\gamma}:=\frac{n_\mu p}{n_\mu-(s-\gamma)p}$. Let $\Omega \subset X$ be a bounded John domain. Let $F\subset\partial\Omega$ be a compact set. Consider a positive increasing function $\phi$ satisfying the growth condition $\phi(2x)\le C \phi(x)$ and such that $w_{\phi}^F\in L^1_{\loc}(\Omega,d\mu)$, and define the function $\Phi(t)=\phi(t)^{p/p^*_{s-\gamma}}$. Then, for any function $u\in W^{s,p}(\Omega,\mu)$,
	\[
	\inf_{a\in \mathbb{R}}\|u-a\|_{L^{ p^*_{s-\gamma}}(\Omega,w_{\phi}^Fd\mu)}\lesssim [u]_{W_{\tau}^{s,p}\left(\Omega,v_{\Phi,\gamma p}^Fd\mu \right)}.
	\]
\end{thm}

\begin{thm}\label{(p,p)} Let  $(X,d, \mu)$ be a metric space with $\mu$ a doubling measure satisfying a $\delta$- reverse doubling property. Consider $w_{\phi}^F$, $F$ and $\phi$ as in the statement of Theorem \ref{P-S inequality}. For $1\leq p<\infty$ and $0<s\leq \delta$ we have the following (p,p) Poincar\'e inequality

	\[
	\inf_{a\in \mathbb{R}}\|u-a\|_{L^ p(\Omega,w_{\phi}^Fd\mu)}\lesssim [u]_{W_{\tau}^{s,p}\left(\Omega,v_{\Phi,s p}^Fd\mu \right)}.
	\]
\end{thm}
In the sequel we will give only the proof of Theorem \ref{P-S inequality} given that the proof of Theorem \ref{(p,p)} follows in the same way, recalling that $I_0^\mu f$ stands for the Hardy-Littlewood maximal operator. The $(1,1)$-inequality follows with this proof from the boundedness of $I_0^\mu$ in $L^\infty$.
\begin{proof}[Proof of Theorem \ref{P-S inequality}]
We proceed by duality. Let $f$ be such that $\|f\|_{L^{({p^*_{s-\gamma}})'}(\Omega, w_{\phi}^Fd\mu)}=1$.  Then, by Lemma \ref{estimacion puntual} and Tonelli's theorem, we have
\begin{equation}\label{acotacion}
\begin{aligned}
&\int_\Omega |u(x)-u_{B_0}|f(x)w_{\phi}^F(x)d\mu(x)\lesssim \int_\Omega I_s^\mu\left(g_p\chi_{\Omega\cap\{d(\cdot,x)\leq c_3d(\cdot)\}}\right)(x) f(x)w_{\phi}^F(x)d\mu(x)\\
&= \int_\Omega \int_{\{x\in\Omega:d(y,x)\leq c_3d(y)\}}\frac{ f(x)  [w_{\phi}^F(x)]^{\frac{1}{({p^*_{s-\gamma}})'}+\frac{1}{{p^*_{s-\gamma}}}}d(x,y)^{s-\gamma+\gamma}}{\mu\left[B(x,d(x,y))\right]}  d\mu(x)g_p(y)d\mu(y).
\end{aligned}
\end{equation}

Now observe that, by hypothesis, $\phi((1+c_3)t)\lesssim \phi(t)$. Hence, using  H\"older's inequality and the boundedness properties of the operator (in the case of Theorem \ref{(p,p)} we also use Lemma \ref{lema maximal} and the boundedness of the Hardy-Littlewood maximal function) we may continue \eqref{acotacion} with
\[
\begin{split}
&\int_\Omega I_{s-\gamma}^\mu\left[f[w_{\phi}^F]^\frac{1}{({p^*_{s-\gamma}})'}\chi_{\Omega\cap\{d(\cdot,y)\leq c_3d(y)\}}\right](y)d(y)^\gamma \phi[d_F(y)]^\frac{1}{p^*_{s-\gamma}} g_p(y)d\mu(y)\\
&\lesssim \left(\int_\Omega \int_{\{x\in \Omega:d(x,y)\leq \tau d(y)\}}\frac{|u(x)-u(y)|^pd(y)^{\gamma p}\phi[d_F(y)]^{\frac{p}{p^*_{s-\gamma}}}}{\mu\left[B(x,d(x,y))\right]d(x,y)^{sp}}d\mu(x)d\mu(y)\right)^{1/p}\\
&\lesssim \left(\int_\Omega \int_{\{x\in \Omega:d(x,y)\leq \tau d(y)\}}\frac{|u(x)-u(y)|^pv_{\Phi,\gamma p}^F(x,y)}{\mu\left[B(x,d(x,y))\right]d(x,y)^{sp}}d\mu(x)d\mu(y)\right)^{1/p}\\
&=[u]_{W_{\tau}^{s,p}\left(\Omega,v_{\Phi,\gamma p}^Fd\mu \right)}.
\end{split}
\]
\end{proof}
\begin{obs}
It is possible to prove the same result whenever $(X,d,\mu)$ is a metric space with an $n_\mu$-doubling and $\delta$-reverse doubling measure $\mu$ with $\eta$-lower Ahlfors-David regularity for some $(s-\gamma)p<\eta$, where $0\leq \gamma<s\leq \delta$ and $p>1$. In this case, the result is obtained with an $L^{p^*_{s-\gamma}}$ norm at the left-hand side, where $p^*_{s-\gamma}:=\frac{\eta p}{\eta-(s-\gamma)p}$. We remark that the $(p,p)$ (which corresponds to Theorem \ref{(p,p)}) does not need the lower Ahlfors-David regularity hypothesis as just the doubling property is needed for the boundedness of the Hardy-Littlewood maximal operator. Nevertheless, we decided to ask for more regularity in order to get cleaner statements. It should be noted that the growth condition on $\phi$ is not  actually necessary, but the results are much cleaner assuming this condition. \end{obs}

  In what follows, we prove the $(1^*_{s-\gamma},1)$-inequality  by requiring some stronger properties on our measure $\mu$. For this, we will use the following lemma, which is a generalization of a well-known result which can be found, for instance, in the book by Jost \cite{J}.
\begin{lem}\label{lemmaJost}
 Take $0<\gamma<s$, $\eta>0$  and $q>1$. Let $(X,d,\mu)$ be a metric space with $\mu$ an $\eta$-upper and $(s-\gamma)q'$-lower Ahlfors-David regular measure. Let $x\in X$ and suppose that for any measurable bounded set $F$ with positive measure there is a ball $B(x,R)$ with comparable measure to that of the set $F$.  Then for any measurable set $E$ with positive measure we have that
 \[
 \int_E\frac{d\mu(y)}{d(x,y)^{(s-\gamma)q'-(s-\gamma)}}\lesssim \mu(E)^{\frac{\eta+s-\gamma}{(s-\gamma)q'}-1}.
 \]
\end{lem}
\begin{proof}
Let $R>0$ be such that the ball $B:=B(x,R)$ verifies $\mu(B)\asymp\mu(E)$.  For this $R$, write 
\[
 \int_E\frac{d\mu(y)}{d(x,y)^{(s-\gamma)q'-(s-\gamma)}}= \left(\int_{E\setminus(E\cap B)}+\int_{E\cap B}\right)\frac{d\mu(y)}{d(x,y)^{(s-\gamma)q'-(s-\gamma)}}
.\]
On one hand, we note that for $y\in E\setminus(E\cap B)$, we have $d(x,y)\geq R$, so
\[
 \int_{E\setminus(E\cap B)}\frac{d\mu(y)}{d(x,y)^{(s-\gamma)q'-(s-\gamma)}}\leq  \int_{E\setminus(E\cap B)}\frac{d\mu(y)}{R^{(s-\gamma)q'-s}}\leq R^{\eta-(s-\gamma)q'+s-\gamma},
\]
as $\mu$ is an $\eta$-upper Ahlfors-David regular measure.

On the other hand, for $y\in B$, we can use Lemma 2.1 in \cite{GCG}, so we obtain
\[
 \int_{E\cap B}\frac{d\mu(y)}{d(x,y)^{(s-\gamma)q'-(s-\gamma)}}\leq 
 \int_{B}\frac{d\mu(y)}{d(x,y)^{(s-\gamma)q'-(s-\gamma)}}\lesssim R^{\eta-(s-\gamma)q'+s-\gamma}.
 \]
 Thus, as $\mu$ is $(s-\gamma)q'$-lower Ahlfors-David regular and $\mu(B)=\mu(E)$, we finally get
  \[
 \int_E\frac{d\mu(y)}{d(x,y)^{(s-\gamma)q'-(s-\gamma)}}\lesssim \mu(E)^{\frac{\eta+s-\gamma}{(s-\gamma)q'}-1}.
 \]
\end{proof}

\begin{obs}
If $\mu$ is an $n_\mu$-Ahlfors-David regular measure, then the space $(X,d,\mu)$ satisfies that for any  bounded measurable set $F$ there exists a ball of comparable size, and, hence, Lemma \ref{lemmaJost} holds. Indeed, for any $x\in X$ it suffices to take the ball $B\left(x,\frac{\mu(F)^{1/n_\mu}}{2c_u^{1/n_\mu}}\right)$. 
\end{obs}

\begin{thm}
Theorem \ref{P-S inequality} also holds for $p=1$.\end{thm}
\begin{proof}

Let us define, for $\lambda>0$, the set $E:=\{x\in \Omega:|u(x)-u_{B_0}|>\lambda\}$ and assume that $\mu$ is $n_\mu$-lower and $n_\mu$-upper Ahlfors-David regular (note that $(s-\gamma)(1^*_{s-\gamma})'=(s-\gamma)\frac{n_\mu}{s-\gamma}=n_\mu$). Then, by Chebyshev's inequality, Lemma \ref{estimacion puntual} and Tonelli's theorem,

\begin{equation}
\begin{aligned}
(w_\phi^F d\mu)(E)&\lesssim \frac{1}{\lambda}\int_E \int_{\Omega\cap\{d(x,y)\leq c_3d(y)\}}\frac{g_1(y) d(x,y)^s}{\mu[B(y,d(x,y))]}d\mu(y) w_\phi^F(x)d\mu(x)\\
&=\frac{1}{\lambda}\int_\Omega g_1(y) \int_{E\cap B(y,c_3 d(y))}\frac{w_\phi^F(x) d(x,y)^s}{\mu[B(y,d(x,y))]}d\mu(x) d\mu(y)\\
&= I_1+ I_2,
\end{aligned}
\end{equation}
where $I_1$ corresponds to the case in which the inner integral  is defined on the region $E_1$ where $d(x,y)\leq \tau d(y)$ and $I_2$ to the case where the inner integral is evaluated on its complement, $E_2$. Observe that when $d(x,y)\leq \tau d(y)$, we have that $(1-\tau) d(y)\leq d(x)\leq (1+\tau)d(y)$ and that the same comparison holds  for $d_F(x)$ and $d_F(y)$, so that, as $\mu$ is $n_\mu$-lower Ahlfors-David regular, by Lemma  \ref{lemmaJost} and the fact that $\frac{n_\mu+s-\gamma}{(s-\gamma)(1^*_{s-\gamma})'}-1=\frac{1}{(1^*_{s-\gamma})'}$,
\[
\begin{split}
I_1 &\lesssim \int_{\Omega}\frac{g_1(y)}{\lambda}\int_{E_1}\frac{d(x,y)^s}{\mu[B(y,d(x,y))]}d\mu(x)w_{\phi}^F(y)d\mu(y)\\
&\lesssim \int_{\Omega}d(y)^\gamma\frac{g_1(y)}{\lambda}\int_{E_1}\frac{d\mu(x)}{d(x,y)^{(s-\gamma)(1^*_{s-\gamma})'-(s-\gamma)}}w_{\phi}^F(y)d\mu(y)\\
&\lesssim  \int_{\Omega}d(y)^\gamma\frac{g_1(y)}{\lambda} \mu(E_1)^{\frac{1}{(1^*_{s-\gamma})'}}w_{\phi}^F(y)d\mu(y)\\
&\lesssim  \int_{\Omega}d(y)^\gamma\frac{g_1(y)}{\lambda} \left(\int_{E_1} w_{\phi}^F(x)d\mu(x)\right)^{\frac{1}{(1^*_{s-\gamma})'}}{\left[w_{\phi}^F(y)\right]}^{\frac{1}{1^*_{s-\gamma}}}d\mu(y)\\
&\lesssim \int_{\Omega}d(y)^\gamma\frac{g_1(y)}{\lambda} (w_{\phi}^Fd\mu)\left({E}\right)^{\frac{1}{(1^*_{s-\gamma})'}}\left[{w_{\phi}^F}(y)\right]^{\frac{1}{1^*_{s-\gamma}}}d\mu(y),
\end{split}
\]
where we have used that, by hypothesis, we know that $\phi((1+\tau)d_F(x))\lesssim \phi(d_F(x))$ and $\phi\left[\frac{(1+\tau)}{1-\tau}d_F(x)\right]\lesssim \phi(d_F(x))$. 

Hence, we have
\[
I_1\lesssim \int_{\Omega}\frac{g_1(y)}{\lambda} (w_{\phi}^Fd\mu)\left({E}\right)^{1/(1^*_{s-\gamma})'} \left[w_{\phi}^F(y)\right]^{1/1^*_{s-\gamma}} d(y)^{\gamma}d\mu(y).
\]

On the other hand, using that $d(x,y)\geq \tau d(y)$, we have that, as $\mu$ is $n_\mu$-upper Ahlfors-David regular and $s\leq n_\mu$,
\[
\begin{split}
I_2&=\frac{1}{\lambda}\int_\Omega g_1(y) \int_{E_2}\frac{w_\phi^F(x) d(x,y)^s}{\mu[B(y,d(x,y))]}d\mu(x) d\mu(y)\\
&\lesssim \frac{1}{\lambda}\int_\Omega g_1(y) d(y)^{s-n_\mu}(w_\phi^F d\mu)(E_2) d\mu(y)\\
&\lesssim \frac{1}{\lambda}\int_\Omega g_1(y) d(y)^{s-n_\mu}(w_\phi^F d\mu)({E_2})^{1/(1^*_{s-\gamma})'}d(y)^{\frac{n_\mu}{1^*_{s-\gamma}}} \left[w_{\phi}^F(y)\right]^{1/1^*_{s-\gamma}} d\mu(y)\\
& \lesssim \frac{1}{\lambda}\int_\Omega g_1(y) d(y)^{\gamma}(w_\phi^F d\mu)({E})^{1/(1^*_{s-\gamma})'}\left[w_{\phi}^F(y)\right]^{1/1^*_{s-\gamma}} d\mu(y).
\end{split}
\]

Thus, we finally get 
\[
(w_\phi^F d\mu)(E)\lesssim\frac{(w_\phi^F d\mu)(E)^{1/(1^*_{s-\gamma})'}}{\lambda}\int_\Omega g_1(y)d(y)^{\gamma} w_{\phi}^F(y)^{1/1^*_{s-\gamma}}  d\mu(y),
\]
i.e. 
\[
\|u-u_{B_0}\|_{L^{1^*_{s-\gamma},\infty}(\Omega,w_\phi^F)}\lesssim \int_\Omega g_1(y)d(y)^{\gamma} w_{{\phi}}^F(y)^{1/1^*_{s-\gamma}}  d\mu(y).
\]
At this point, a ``weak implies strong'' argument, which also holds in our setting (see the comments preceding \cite[Lemma 3.2.]{DD2} and also \cite[Proposition 5]{DV}, \cite[Theorem 4]{Ha2}, \cite{DIV}) gives us the extremal case $p=1$ with weight $w_\phi^F$ at the left-hand side and $v_{\Phi,\gamma}^F$ at the right hand side.

\end{proof}
\begin{obs}
In this case, Ahlfors-David regularity is needed for the argument, so Lemma \ref{lemmaJost} can be applied without any other assumption.
\end{obs}

\section{Some particular cases of Theorems \ref{P-S inequality} and \ref{(p,p)}}\label{ejemplos}
In this section we will give several instances of Theorems \ref{P-S inequality} and \ref{(p,p)} by making specific choices of $(X,d,\mu)$ and $\phi$. These particular examples show that  Theorems \ref{P-S inequality} and \ref{(p,p)} extend results in \cite{DD2,HV,LH} in several aspects. 

Let us start with the Euclidean space with the Lebesgue measure, $(\mathbb{R}^n,d,|\cdot|)$, which is a doubling measure space with $n$-Ahlfors-David regularity. If we choose $\phi(t)=t^a$, where $a\geq 0$, $\Omega$ any bounded John domain in $\mathbb{R}^n$ and $F=\partial\Omega$, then we recover the results in  \cite{DD2} about John domains. More precisely, 

\begin{cor}[Theorems 3.1. and 3.2. in \cite{DD2}]\label{Euclidean case1}
Let $\Omega$ be a bounded John domain in $\mathbb{R}^n$. Let $\tau\in(0,1)$ and $a\geq0$. Let $s\in (0,1)$ and take $1\le p<\infty$ such that $ps< n$. Thus, for any $q\leq p^*_{s}=\frac{pn}{n-sp}$, we have that, for any function $u\in W^{s,p}(\Omega,dx)$,
\[
\begin{split}
\inf_{c\in \mathbb{R}}\|u-c&\|_{L^{p^*_{s}}(\Omega,d^a)}\\
&\lesssim\left(\int_\Omega\int_{\{z\in \Omega:d(z,y)\leq \tau d(y)\}}\frac{|u(z)-u(y)|^p}{|z-y|^{n+sp}}\delta(z,y)^{b}dzdy\right)^{\frac{1}{p}},
\end{split}
\]
where $\delta(z,y)=\min_{x\in\{z,y\}}d(x)$ and  $b\leq a\frac{p}{q}+s p$.
\end{cor}

Moreover, by choosing $F\subsetneq \partial\Omega$ in Theorem \ref{(p,p)} we recover the main result in \cite{LH}, namely
\begin{cor}[Theorem 1.1 in \cite{LH}]\label{Euclidean case2}
Let $\Omega$ in $\mathbb{R}^n$ be a bounded John domain and $1<p<\infty$. Given $F$ a compact set in $\partial\Omega$, and the parameters $s,\tau\in(0,1)$ and $a\geq0$, the inequality 
\[
\begin{split}
\inf_{c\in \mathbb{R}}\|u-c&\|_{L^{p}(\Omega,d_F^a)}\\
&\lesssim\left(\int_\Omega\int_{\{z\in \Omega:d(z,y)\leq \tau d(y)\}}\frac{|u(z)-u(y)|^p}{|z-y|^{n+sp}}\delta^{s p}(z,y)\delta_F(z,y)^{a}dzdy\right)^{\frac{1}{p}},
\end{split}
\]
holds for any function $u\in W^{s,p}(\Omega,dx)$, where $\delta_F(z,y)=\min_{x\in\{z,y\}}d_F(x)$. 
\end{cor}

If we use Theorem \ref{P-S inequality} for $F\subsetneq \partial\Omega$, then we improve both results by obtaining the following combination of them:

\begin{cor}\label{combination}
Let $\Omega$ be a bounded John domain in $\mathbb{R}^n$ and consider $F\subset\partial \Omega$. Let $\tau\in(0,1)$ and $a\geq0$. Let $s\in (0,1)$ and take $1\le p<\infty$ such that $ps< n$. Thus, for any $q\leq p^*_{s}=\frac{pn}{n-sp}$, we have that, for any function $u\in W^{s,p}(\Omega,dx)$,
\[
\begin{split}
&\inf_{c\in \mathbb{R}}\|u-c\|_{L^{p^*_{s}}(\Omega,d_F^a)}\\
&\lesssim\left(\int_\Omega\int_{\{z\in \Omega:d(z,y)\leq \tau d(y)\}}\frac{|u(z)-u(y)|^p}{|z-y|^{n+sp}}\delta(z,y)^{sp}\delta_F(z,y)^{b}dzdy\right)^{\frac{1}{p}},
\end{split}
\]
where $\delta_F(z,y)=\min_{x\in\{z,y\}}d_F(x)$ and  $b\leq a\frac{p}{q}$.
\end{cor}

In general, we are able to include in our inequalities a large class of weights defined by using the distance from the boundary. An instance of weights which is not included in the previous results is, for example, the family of weights $w_\phi^F$, where $\phi(t)=t^a\log^b(e+ t)$, $a,b\geq 0$.

Another space on which our results can be applied is the Heisenberg group $\mathbb{H}^n$, as it is a  $(2n+2)$-Ahlfors-David regular metric measure space endowed with the Carnot-Carath\'eodory metric $d:=d_{CC}$ (see \cite{S2} for the definition of $d_{CC}$) and the Lebesgue measure. The result reads as follows:

\begin{cor}\label{Heisenberg case}
Let us consider the Heisenberg space $(\mathbb{H}^n,d_{CC},\mathcal{H}^{2n+2})$, where $d_{CC}$ is the Carnot-Carath\'eodory metric and $\mathcal{H}^{2n+2}$ is the $(2n+2)$-dimensional Hausdorff measure on $\mathbb{H}^n$. Let $\Omega$ be a bounded John domain in $\mathbb{H}^n$ and $F\subset \partial\Omega$ a compact set. Let $s,\tau\in(0,1)$ and $a\geq0$ such that $d_F^a\in L^1_{\loc}(\Omega)$. Let $0<\gamma \leq s$  and take $p>1$ such that $(s-\gamma)p< 2n+2$. Thus, for any $q\leq p^*_{s-\gamma}=\frac{p(2n+2)}{(2n+2)-(s-\gamma)p}$, we have that, for any function $u\in W^{s,p}(\Omega,\mathcal{H}^{2n+2})$,
\[
\begin{split}
&\inf_{c\in \mathbb{R}}\|u-c\|_{L^{p^*_{s-\gamma}}(\Omega,d_F^a\mathcal{H}^{2n+2})}^p\\
&\lesssim\int_\Omega\int_{\Omega\cap B(y,\tau d(y))}\frac{|u(z)-u(y)|^p\delta_F^{\gamma p}(z,y)\delta(z,y)^{a \frac{p}{q}}d\mathcal{H}^{2n+2}(z)d\mathcal{H}^{2n+2}(y)}{{(|z_1-y_1|^4+|z_2-y_2+2\Im \langle z_1,y_1\rangle|^2)^{\frac{2n+2+sp}{4}}}},
\end{split}
\]
where $\delta_F(z,y)=\min_{(x)\in\{z,y\}}{d}_F(x)$, with $d_F$ the Carnot-Carath\'eodory distance from $x$ to $\partial\Omega$. 
If we choose $F=\partial\Omega$, then we can obtain the inequality with the power $d^b$ at the right hand side for the range $b\leq a\frac{p}{q}+\gamma p$.
\end{cor}
In general, for a given metric space $X$, our result can be applied for any metric space endowed with the $\alpha$-dimensional Hausdorff measure, as every $\alpha$-Ahlfors-David regular measure is comparable to the $\alpha$-dimensional Hausdorff measure. Also note that our results allow to obtain improved fractional Poincar\'e-Sobolev inequalities for any ball in a metric measure space satisfying our conditions when the metric of the space is the Carnot-Carath\'eodory metric, as balls in these spaces are Boman chain domains.

\section{Sufficient conditions for a bounded domain} \label{suficiencia del dominio}
In this section we extend Theorem 3.1 in \cite{HV} to the more general context of doubling metric measure spaces, improving it by including weights. As an example, we obtain sufficient conditions for a domain in $(X,d,\mu)$ to support the classical improved $(q,p)$-Poincar\'e inequality.  We refer the reader to Remark \ref{dilatadas y cadenas} for the basic definitions concerning chains of balls of a Whitney decomposition of a domain $\Omega$.

First of all, we  prove an unweighted fractional $(q,p)$-Poincar\'e inequality on balls. This lemma was first proved in the Euclidean case in \cite[Lemma 2.2]{HV}.

\begin{lem}\label{lema2.2}
Let $B$ be a ball in a metric space $(X,d,\mu)$ with a doubling measure $\mu$. Let $1\leq q\leq p<\infty$ and let $s,\rho\in(0,1)$. Then, \begin{equation}
\begin{aligned}
&\int_B|u(y)-u_{B}|^qd\mu(y)\\
&\lesssim\frac{r(B)^{s q}}{\mu(B)^{\frac{q-p}{p}}}\left(\int_{B^*}\int_{\{z\in B^*:d(z,y)\leq \rho r(B)\}}\frac{|u(y)-u(z)|^pd\mu(z)d\mu(y)}{\mu[B(z,d(z,y))]d(z,y)^{s p}} \right)^{q/p},
\end{aligned}\end{equation}
for any $u\in L^p_\mu(B)$, where $B^*$ is defined as in Remark \ref{dilatadas y cadenas}.
\end{lem}
\begin{proof}
Let us consider a covering $\mathcal{B}=\{B_i\}_{i\in J}$ of $B$ by $J$ balls of radious $\frac{\rho}{K}r(B)$ for some $K>1$. This can be done in such a way that, when $R$ is the union of two balls $B_i$ and $B_j$ with overlapping dilations  (i.e. with  $\lambda B_i\cap \lambda B_j\neq \emptyset$ for some $\lambda>1$ sufficiently small), $R\subset B(y,\rho r(B))$ for every $y\in R$. Also, such an $R$ satisfies $R\subset B^*$ and $\mu[B(z,d(z,y))]\lesssim\mu(R)$. Observe that the index set is uniformly finite for every ball $B$, as $X$ is a doubling metric space.

Once we have this construction, observe that, for the union $R$ of two balls in $\mathcal{B}$ with overlapping dilations, we have, by the doubling condition
\[
\begin{split}
\frac{1}{\mu(R)}\int_R&|u(y)-u_{R}|^qd\mu(y)\leq \left(\frac{1}{\mu(R)}\int_R|u(y)-u_{R}|^pd\mu(y)\right)^{q/p}\\
&\leq \left(\frac{1}{\mu(R)}\int_R\frac{1}{\mu(R)}\int_R|u(y)-u(z)|^pd\mu(z)d\mu(y)\right)^{q/p}
\\
&\lesssim \left(\frac{1}{\mu(R)}\int_R\int_R\frac{|u(y)-u(z)|^p\diam(R)^{s p}d\mu(z)d\mu(y)}{\mu[B(z,d(z,y))]d(z,y)^{s p}}\right)^{q/p}\\
&\lesssim \frac{r(B)^{s q}}{\mu(B)^{q/p}}\left(\int_{B^*}\int_{B^*\cap B(y,\rho r(B))}\frac{|u(y)-u(z)|^p d\mu(z)d\mu(y)}{\mu[B(z,d(z,y))]d(z,y)^{s p}}\right)^{q/p}.
\end{split}
\]

With this in mind, observe that, by H\"older's and Minkowski's,
\[
\begin{split}
\frac{1}{\mu(B)}\int_B|u(y)-u_{B}|^qd\mu(y)&\lesssim \frac{1}{\mu(B)}\int_B |u(y)-u_{B_1}|^qd\mu(y)\\
&\lesssim \sum_{j\in J}\frac{1}{\mu(B_j)}\int_{B_j}|u(y)-u_{B_j}|^qd\mu(y)\\
&\qquad+ \sum_{j\in J}\frac{1}{\mu(B_j)}\int_{B_j}|u_{B_j}-u_{B_1}|^qd\mu(y).
\end{split}
\]

The first sum is bounded by the quantity above, so it is enough to estimate the second sum. In order to do this, let us fix $B_j$, $j\in J$ and let $\sigma:\{1,2,\ldots,l\}\to J$, $l\leq \# J$ an injective map such that $\sigma(1)=1$ and $\sigma (l)=j$, and the subsequent balls $B_{\sigma(i)}$ and $B_{\sigma(i+1)}$ have overlapping dilations. Since $l\leq \# J$, we obtain
\[
\begin{split}
|u_{B_j}-u_{B_1}|^q&\leq \left(\sum_{i=1}^{l-1}|u_{B_{\sigma(i+1)}}-u_{B_{\sigma(i)}}|\right)^{q}\\
&\lesssim\sum_{i=1}^{l-1}|u_{B_{\sigma(i+1)}}-u_{B_{\sigma(i+1)}\cup B_{\sigma(i)}}|^q\\
&\qquad\qquad+ \sum_{i=1}^{l-1}|u_{B_{\sigma(i+1)}\cup B_{\sigma(i)}}-u_{B_{\sigma(i)}}|^q.
\end{split}
\]
The two sums above can be bounded in the same way, so we will just work with the first one. For each term we have
\[
\begin{split}
|u_{B_{\sigma(i+1)}}-&u_{B_{\sigma(i+1)}\cup B_{\sigma(i)}}|^q\\
&= \frac{1}{\mu(B_{\sigma(i+1)})}\int_{B_{\sigma(i+1)}}|u_{B_{\sigma(i+1)}}-u+u-u_{B_{\sigma(i+1)}\cup B_{\sigma(i)}}|^qd\mu\\
&\lesssim \frac{1}{\mu(B_{\sigma(i+1)})}\int_{B_{\sigma(i+1)}}|u-u_{B_{\sigma(i+1)}}|^qd\mu\\
&\qquad \qquad +\frac{1}{\mu(B_{\sigma(i+1)}\cup B_{\sigma(i)})}\int_{B_{\sigma(i+1)}\cup B_{\sigma(i)}}|u-u_{B_{\sigma(i+1)}\cup B_{\sigma(i)}}|^qd\mu,
\end{split}
\]
where we have used the conditions on the union of two balls of the covering with nonempty intersection and the doubling condition.
In the last two integrals we can apply the first estimate above to obtain the desired result.
\end{proof}

With this lemma at hand, we can proceed to the proof of Theorem \ref{sufi},  which gives sufficient conditions on a domain of a doubling metric space to support $(w_\phi,v_{\Phi,\gamma p})$-improved fractional $(q,p)$-Poincar\'e inequalities for $q\leq p$ and suitable functions $\phi$ and $\Phi$.

\begin{thm}\label{sufi}
	Let $\Omega$ be a domain in a doubling metric space $(X,d)$ with a Whitney decomposition $W_M$ as the one built in Lemma \ref{Whitney} and with the properties in Remark \ref{dilatadas y cadenas}. Let $\phi$ be a positive increasing function satisfying the growth condition $\phi(2x)\le C \phi(x)$.  Let $1\leq q\leq p<\infty$, and let $s,\tau\in(0,1)$ and $0\leq \gamma\leq s$.
	\begin{enumerate}
		\item If there exists a chain decomposition of $\Omega$ such that
		\begin{equation}
		\sum_{E\in W_M}\left(\sum_{B\in E(W_M)}r(E)^{(s-\gamma) q}\frac{\phi(r(B))}{\phi(r(E))}\ell[\mathcal{C}(B^*)]^{q-1}\frac{\mu(B)}{\mu(E)^{q/p}}\right)^{p/(p-q)}<\infty,
		\end{equation}
		then $\Omega$ supports the $(w_\phi,v_{\Phi,\gamma p})$-improved fractional $(q,p)$-Poincar\'e inequality
		\[
		\inf_{a\in \mathbb{R}}\|u-a\|_{L^{q}(\Omega,w_{\phi}d\mu)}\lesssim [u]_{W_{\tau}^{s,p}\left(\Omega,v_{\Phi,\gamma p}d\mu \right)},
		\]
		where $\Phi(t)=\phi^{\frac{p}{q}}(t)$.
		\item If $q=p$, and if there exists a chain decomposition of $\Omega$ such that 
		\begin{equation}\label{suficp}
		\sup_{E\in W_M}\sum_{B\in E(W_M)}r(E)^{(s-\gamma) p}\frac{\phi(r(B))}{\phi(r(E))}\ell[\mathcal{C}(B^*)]^{p-1}\frac{\mu(B)}{\mu(E)}<\infty,
		\end{equation}
		then $\Omega$ supports the $(w_\phi,v_{\phi,\gamma p})$-improved fractional $(p,p)$-Poincar\'e inequality
		\[
		\inf_{a\in \mathbb{R}}\|u-a\|_{L^{p}(\Omega,w_{\phi}d\mu)}\lesssim [u]_{W_{\tau}^{s,p}\left(\Omega,v_{\phi,\gamma p}d\mu \right)}.
		\]

	\end{enumerate}
\end{thm}

\begin{proof}
We just prove the first statement, as the second one follows in the same way. We can use H\"older's, Minkowski's and the Whitney decomposition of $\Omega$ to obtain
\begin{equation}\label{th31}
\begin{aligned}
 \int_{\Omega}|u(x)-u_{B_0^*}|^qw_\phi(x)d\mu(x)
&\lesssim \sum_{B\in W_M}\int_{B^*}|u(x)-u_{B_0^*}|^qw_\phi(x)d\mu(x)\\
&\lesssim \sum_{B\in W_M}\int_{B^*}|u(x)-u_{B^*}|^qw_\phi(x)d\mu(x)\\
&\qquad\qquad+\sum_{B\in W_M}\int_{B^*}|u_{B^*}-u_{B_0^*}|^qw_\phi(x)d\mu(x).
\end{aligned}\end{equation}

We begin by estimating the first sum. Using Lemma \ref{lema2.2} with $\rho=C_M\tau$ (where $C_M<1$ is such that $\rho r(B^*)\leq \tau d(x)$ for any $x\in B^{**}$) and the fact that $d(x)\asymp r(B^*)$ for any $x\in B^*$ (and the corresponding fact for $B^{**})$, then, taking into account the choice of $M$ and the growth condition on $\phi$, we obtain
\[
\begin{split}
&\int_{B^*}|u(x)-u_{B^*}|^qw_\phi(x) d\mu(x)\lesssim \frac{r(B)^{(s-\gamma) q}}{\mu(B^*)^{q/p-1}}[u]_{W_{\tau}^{s,p}(B^{**},w_{\Phi,\gamma p})}^q.
\end{split}
\]
Thus,
\[
\begin{split}
\sum_{B\in W_M} \int_{B^*}|u(x)&-u_{B^*}|^qw_\phi(x)d\mu(x)
\lesssim \sum_{B\in W_M} \frac{r(B)^{(s-\gamma) q}}{\mu(B^*)^{q/p-1}}[u]_{W_{\tau}^{s,p}(B^{**},w_{\Phi,\gamma p})}^q\\
&\leq
\left(\sum_{B\in W_M} \mu(B^*)\right)^{\frac{p-q}{p}}\left( \sum_{B\in W_M}r(B)^{(s-\gamma) p}[u]_{W_{\tau}^{s,p}(B^{**},w_{\Phi,\gamma p})}^p\right)^{q/p}\\
&\lesssim \mu(\Omega)^{\frac{p-q}{p}}\left( \sum_{B\in W_M}r(B)^{(s-\gamma)p}[u]_{W_{\tau}^{s,p}(B^{**},w_{\Phi,\gamma p})}^p\right)^{q/p}\\
&\lesssim\mu(\Omega)^{\frac{p-q}{p}}\diam(\Omega)^{(s-\gamma) q}[u]_{W_{\tau}^{s,p}(\Omega,v_{\Phi,\gamma p})}^{q},
\end{split}
\]
where we have used that $\{B^*\}_{B\in W_M}$ and $\{B^{**}\}_{B\in W_M}$ are families with uniformly bounded overlapping contained in $\Omega$ and also that, in the domain of integration, the distance from each variable to the boundary of $\Omega$ is comparable to the other.

Next, we estimate the second sum in \eqref{th31}. By using chains of the decomposition and again that  $d(x)\asymp r(B^*)$, $x\in B^*$,
\[
\begin{split}
\sum_{B\in W_M}\int_{B^*}|u_{B^*}-u_{B_0^*}|^q&w_\phi(x)d\mu(x)\lesssim \sum_{B\in W_M}\mu(B)\phi(r(B))\left(\sum_{j=1}^k|u_{B_j^*}-u_{B_{j-1}^*}|\right)^q\\
&\leq \sum_{B\in W_M}\ell[\mathcal{C}(B^*)]^{q-1}\mu(B)\phi(r(B))\left(\sum_{j=1}^k|u_{B_j^*}-u_{B_{j-1}^*}|^q\right).
\end{split}
\]
Since $\max\{\mu(B_j^*),\mu(B_{j-1}^*)\}\lesssim \mu(B_j^*\cap B_{j-1}^*)$ (see \eqref{chainprop}), we can write, by using H\"older's inequality,
\[
\begin{split}
|u_{B_j^*}-u_{B_{j-1}^*}|^q&\lesssim \sum_{i=j-1}^j\left(\mu(B_i^*)^{-1}\int_{B_i^*}|u(x)-u_{B_i^*}|d\mu(x)\right)^q\\
&\leq \sum_{i=j-1}^j\mu(B_i^*)^{-1}\int_{B_i^*}|u(x)-u_{B_i^*}|^qd\mu(x), 
\end{split}
\]
where, for the first inequality, we have used that
\[
\begin{split}
|u_{B_j^*}-u_{B_{j-1}^*}|^q=\left|\frac{1}{\mu(B_j^*\cap B_{j-1}^*)}\int_{B_j^*\cap B_{j-1}^*}(u_{B_j^*}-u(x)+u(x)-u_{B_{j-1}^*})d\mu(x)\right|^q\\
\lesssim \sum_{i=j-1}^j\frac{1}{\mu(B_j^*\cap B_{j-1}^*)}\int_{B_j^*\cap B_{j-1}^*}|u_{B_i^*}-u(x)|^q.
\end{split}
\]

A new application of Lemma \ref{lema2.2} gives
\[
\begin{split}
|u_{B_j^*}-u_{B_{j-1}^*}|^q \lesssim \sum_{i=j-1}^j \frac{r(B_i)^{(s-\gamma) q}}{\phi(r(B_i))}\mu(B_i^*)^{-q/p}  [u]_{W_{\tau}^{s,p}(B_i^{**},w_{\Phi,\gamma p})}^q,
\end{split}
\]
so the second sum in \eqref{th31} is bounded by the sum
\[
\begin{split}
\sum_{{B}\in W_M}\phi(r(B))\ell[\mathcal{C}&(B^*)]^{q-1}\mu(B)\left[\sum_{j=0}^k \frac{r(B_j)^{(s-\gamma)q}}{\phi(r(B_j))}\mu(B_j^*)^{-q/p}[u]_{W_{\tau}^{s,p}(B_j^{**},w_{\Phi,\gamma p})}^q \right].
\end{split}
\]
Rearranging the sum as in \cite{HV}, we get 
\[
\begin{split}
&\sum_{B\in W_M}\int_{B^*}|u_{B^*}-u_{B_0^*}|^qd\mu(x)\\&\lesssim \sum_{E\in W_M}\sum_{B\in E(W_M)}r(E)^{(s-\gamma) q}\frac{\phi(r(B))}{\phi(r(E))}\ell[\mathcal{C}(B^*)]^{q-1}\frac{\mu(B)}{\mu(E)^{q/p}} [u]_{W_{\tau}^{s,p}(E^{**},w_{\Phi,\gamma p})}^q.
\end{split}
\]
Now, H\"older's inequality together with the hypothesis allow us to bound the sum above by $[u]_{W_{\tau}^{s,p}(\Omega,v_{\Phi,\gamma p})}^q$ times the following expresion
\[
\begin{split}
&\left[\sum_{E\in W_M}\left(\sum_{B\in E(W_M)}r(E)^{(s-\gamma) q}\frac{\phi(r(B))}{\phi(r(E))}\ell[\mathcal{C}(B^*)]^{q-1}\frac{\mu(B)}{\mu(E)^{q/p}}\right)^{\frac{p}{p-q}}\right]^{\frac{p-q}{p}} ,
\end{split}
\]
and the result follows.
\end{proof}
\begin{obs}
Observe that, in the case $0\leq\gamma<s$, the constant in the obtained Poincar\'e inequality depends on the size of the domain $\Omega$. This also happens if one thinks of a nonimproved version of the result, as one can check in the proof of Theorem 3.1 in \cite{HV}.
\end{obs}

\section{An example of application of Theorem \ref{sufi}: the case of John domains in a complete metric space} \label{ejemplo suficiencia}

In this section we will prove a positive result for John domains in complete doubling metric spaces. We begin with a generalization of Lemma 2.7 in \cite{HV}. To this end, we will use the dyadic structure given by Hyt\"onen and Kairema and introduced in Section \ref{preliminares}. We will also use the concept of porous sets in a metric space.
\begin{defi}
A set $S$ in a metric space $(X,d)$ is porous in $X$ if for some $\kappa\in\left(0,1\right]$ the following statement is true: for every $x\in X$ and $0<r\leq 1$ there is  $y\in B(x,r)$ such that $B(y,\kappa r)\cap S=\emptyset$.
\end{defi}
Examples of porous sets are the boundaries of bounded John domains in a complete metric space (see \cite{MV} for the result in the Euclidean case and observe that Ascoli-Arzela's  theorem allows us to prove the result in a complete metric space).

\begin{lem}\label{porous}
Let  $(X,d,\mu)$ be a metric space endowed with a doubling measure, and let $S\subset X$ be a porous set. Let $1\leq p<\infty$. If $x\in S$ and $0<r\leq 1$, then
\[\int_{B(x,r)}\log^p\frac{1}{d(y,S)}d\mu(y)\leq c\mu(B(x,r))\left(1+\log^p\frac{1}{r}\right),\]
where the constant $c$ is independent of $x$ and $r$.
\end{lem}

\begin{proof}
We follow here the notations of Theorem \ref{dyadiccubes}. Let us define
\[\mathcal{C}_{S,\gamma}:=\{Q\in\mathcal{D}:\gamma^{-1}d(z_Q,S)\leq r(b(Q))\leq r(B(Q))\leq 1\},\]
where $z_Q$ is the center of the two balls $b(Q)$ and $B(Q)$ associated to the dyadic cube $Q$.

Suppose $R\in \mathcal{D}$ is a dyadic cube such that $r(B(R))\leq 1$ and such that $d(y,S)\leq 2\left(\frac{C_1}{c_1s}\right)^2r(b(R))$ for some $y\in R$. Then, 
\begin{equation}\label{28}
\begin{aligned}
d(z_R,S)&\leq  d(z_R,y)+d(y,S)\leq r(B(R))+2\left(\frac{C_1}{c_1s}\right)^2r(b(R))\\
&\leq \frac{C_1}{c_1}r(b(R))+2\left(\frac{C_1}{c_1s}\right)^2r(b(R))
= \left(\frac{C_1}{c_1}+2\left(\frac{C_1}{c_1s}\right)^2\right)r(b(R)),
\end{aligned}
\end{equation}

so we have $R\in \mathcal{C}_{S,\gamma}$ for $\gamma =\frac{C_1}{c_1}+2\left(\frac{C_1}{c_1s}\right)^2$.

Fix $j$ a nonnegative integer such that $C_1s^{j+1}\leq r< C_1s^{j}$, and consider a dyadic cube $Q\in \mathcal{D}_{j+1}$ for which $Q\cap B(x,r)\neq \emptyset$. We can cover $B(x,r)$ with cubes like this, as $\mathcal{D}_{j+1}$ partitions $X$, so it will be enough to prove that, for any of these cubes $Q$, we can get
\[ \left\|\log \frac{1}{d(\cdot,S)}\right\|_{L^p_\mu(Q\cap B(x,r))}^p\lesssim \mu(Q\cap B(x,r))\left(1+\log^p\frac{1}{r}\right),\]
as we are working in a doubling metric space, which implies that the size of any covering of $B$ by cubes of the size stated above, is uniformly bounded.

By the porosity of $S$ we know that, since $\mu$ is doubling, $S$ has zero $\mu$-measure \cite[Proposition 3.4]{JJKRRS}, so it is enough to consider points $y\in Q\cap B(x,r)\backslash S$.  Since $x\in S$, we have that
\[
d(y,S)\leq d(y,x)\leq r\leq {C_1} s^{j}\leq C_1(c_1s)^{-1}\diam Q, 
\]
i.e.,
\begin{equation}\label{210}1\leq C_1(c_1s)^{-1}\frac{\diam Q}{d(y,S)}.
\end{equation}
Consider now a sequence of dyadic cubes $Q=Q_0(y)\supset Q_1(y)\supset\cdots\supset Q_m(y),$ each of them containing $y$ and $Q_i(y)$ and $Q_{i+1}(y)$ being immediate ancestor and son, respectively. This, in particular, means that
\begin{equation}\label{211}
\frac{2C_1}{c_1}s^{-1}\geq \frac{\diam B(Q_i(y))}{\diam b(Q_{i+1}(y))}\geq\frac{\diam Q_i(y)}{\diam Q_{i+1}(y)}\geq \frac{\diam b(Q_i(y))}{\diam B(Q_{i+1}(y))}\geq \frac{c_1}{2C_1}s^{-1}.
\end{equation}

We choose $m$ such that the last cube in the sequence satisfies
\begin{equation}\label{212}
\frac{c_1 s^2}{C_1}d(y,S)\leq \diam Q_m(y)< \frac{c_1s}{C_1} d(y,S).
\end{equation}
From \eqref{210} it follows that $m\geq 1$, and by \eqref{211} and \eqref{212},
\[
\left(\frac{2C_1}{c_1}s^{-1}\right)^m\geq \prod_{i=0}^{m-1}\frac{\diam Q_i(y)}{\diam Q_{i+1}(y)}=\frac{\diam Q_0(y)}{\diam Q_m(y)}>C_1(c_1s)^{-1}\frac{\diam Q}{d(y,S)}\geq 1.
\]
Thus, 
\begin{equation}\label{213}
\begin{aligned} m&\geq \log_{\frac{2C_1}{c_1s}}\left(\frac{2C_1}{c_1s}\right)^m
=\frac{1}{\log \left(\frac{2C_1}{c_1s}\right)}\log \left(\frac{2C_1}{c_1s}\right)^m\\
&\geq\frac{1}{\log \left(\frac{2C_1}{c_1s}\right)}\left[ \log\frac{C_1\diam Q}{c_1s}-\log d(y,S)\right]\geq0. 
\end{aligned}
\end{equation}
Furthermore, \eqref{212} and \eqref{28} yield $Q_i(y)\in \mathcal{C}_{S,\gamma}$ for $i=0,1,\ldots,m$ when $\gamma$ is as in \eqref{28}. Thus, we obtain
\[
\begin{split}
\sum_{\underset{R\subset Q}{R\in \mathcal{C}_{S,\gamma}}} \chi_R(y)&\geq 1+m\geq 1+\frac{1}{\log \left(\frac{2C_1}{c_1s}\right)}\left[ \log\frac{C_1\diam Q}{c_1s}-\log d(y,S)\right] \\
&\geq \frac{1}{\log \left(\frac{2C_1}{c_1s}\right)}\left[ \log \left(\frac{2C_1}{c_1s}\right)+\log\frac{C_1\diam Q}{c_1s}-\log d(y,S) \right]\geq0,
\end{split}
\]
as $2C_1>s c_1$.

If we now integrate and apply triangle inequality, we get
\[\begin{split}
&\left\|\log\frac{1}{d(\cdot,S)}\right\|_{L^{p}_\mu(Q\cap B(x,r))}\\
&\leq \left|\log \left(\frac{2C_1^2\diam Q}{(c_1s)^2}\right)\right|\mu(Q\cap B(x,r))^{\frac{1}{p}}+\log \left(\frac{2C_1}{c_1 s}\right)\left\|\sum_{\underset{R\subset Q}{R\in \mathcal{C}_{S,\gamma}}} \chi_R\right\|_{L^{p}_\mu(Q\cap B(x,r))}\\
&\lesssim\left(1+\log\frac{C_1}{c_1s\diam Q}\right)\mu(Q\cap B(x,r))^{\frac{1}{p}}+\left\|\sum_{\underset{R\subset Q}{R\in \mathcal{C}_{S,\gamma}}} \chi_R\right\|_{L^{p}_\mu(Q\cap B(x,r))}\\
&=\left(1+\log\frac{C_1}{c_1s \diam Q}\right)\mu(Q\cap B(x,r))^{\frac{1}{p}}+\left\|\sum_{\underset{R\subset Q}{R\in \mathcal{C}_{S,\gamma}}} \chi_R\right\|_{L^{p}_\mu(Q\cap B(x,r))},
\end{split}\]
where we have used that $2C_1>ec_1s$ (this is immediate by taking into account the definitions of $C_1,c_1$ and $s$ in Theorem \ref{dyadiccubes}), and the fact that, in particular, $\diam Q\leq C_1s^{j+1}\leq r<1$.
Now observe that 
\[\diam Q\geq c_1s^{j+1}=\frac{c_1s}{C_1}C_1s^j>\frac{c_1s}{C_1}r,\]
so 
\[\begin{split}
&\left\|\log\frac{1}{d(\cdot,S)}\right\|_{L^{p}_\mu(Q\cap B(x,r))}\\
&\lesssim \left(1+\log\frac{1}{r}\right)\mu(Q\cap B(x,r))^{1/p}+\left\|\sum_{\underset{R\subset Q}{R\in \mathcal{C}_{S,\gamma}}} \chi_R\right\|_{L^{p}_\mu(Q\cap B(x,r))}.
\end{split}\]

Since $S$ is porous in $X$, we can follow the proof of \cite[Theorem 2.10]{IV} in our context. We obtain a finite positive constant $K_\kappa$, depending on the porosity constant $\kappa$, and families 
\[
\{\hat{R}\}_{R\in \mathcal{C}_{S,\gamma}^k},\quad \mathcal{C}_{S,\gamma}^k\subset \mathcal{C}_{S,\gamma},\qquad k=0,1,\ldots,K_\kappa-1,
\]
where each $\{\hat{R}\}_{R\in \mathcal{C}_{S,\gamma}^k}$ is a disjoint family of cubes $\hat{R}\subset R$, such that
\[\begin{split}
\left\|\sum_{\underset{R\subset Q}{R\in \mathcal{C}_{S,\gamma}}} \chi_R\right\|_{L^{p}_\mu(Q\cap B(x,r))}&\lesssim \sum_{k=0}^{K_\kappa-1}\left\|\sum_{\underset{R\subset Q}{R\in \mathcal{C}^k_{S,\gamma}}} \chi_{\hat{R}}\right\|_{L^{p}_\mu(Q\cap B(x,r))}\\
&\leq \sum_{k=0}^{K_\kappa-1}\|\chi_Q\|_{L^p_\mu(Q\cap B(x,r))}\lesssim \mu(Q\cap B(x,r))^{1/p}.
\end{split}
\]
Hence, we get
\[\begin{split}
\left\|\log\frac{1}{d(\cdot,S)}\right\|_{L^{p}_\mu(Q\cap B(x,r))}
&\lesssim \left(1+\log\frac{1}{r}\right)\mu(Q\cap B(x,r))^{1/p},\end{split}\]
which finishes the proof.
\end{proof}

Following  \cite{HV} it is possible to construct a chain decomposition of a given John domain $\Omega$ on a metric space in such a way that, for a given ball $B\in W_M$ in a Whitney covering of the domain, 
we have that the chain associated to $B$ satisfies 
\begin{equation}\label{cadenas}
\begin{aligned}
\ell[\mathcal{C}(B^*)]&\lesssim\left(1+\log\frac{1}{\diam B}\right). 
\end{aligned}
\end{equation}
Using this result and the fact that a John domain in a complete metric doubling space has porous boundary, we can prove an $(w_\phi,v_{\phi,\gamma p})$-improved fractional $(p,p)$-Poincar\'e inequality on John domains.
\begin{thm}
  Let $1\leq p<\infty$,  $\tau,s\in(0,1)$ and $0\leq \gamma< s$. Let $\phi$ be a positive increasing function and define $w_\phi$ and $v_{\phi,\gamma p}$ as in Theorem \ref{sufi}. A John domain $\Omega$ in a complete doubling metric space $(X,d,\mu)$ supports the $(w_\phi,v_{\phi,\gamma p})$-improved fractional $(p,p)$-Poincar\'e inequality.
\end{thm}
\begin{proof} 
We may assume $\diam \Omega\leq1$. 
We will check condition \eqref{suficp} of Theorem \ref{sufi}. If $E$ is a ball in $W_M$, then
\[\bigcup_{B\in E(W_M)}B\subset B(\omega_E,\min\{1, c\diam(E)\}),\]
where $\omega_E$ is the closest point in $\partial\Omega$ to $x_E$ and $c$ is a positive constant independent of $E$. This follows from the fact that, if $B\in E(W_M)$, then $E$ is closer to $x_0$ than $B$, so $B$ is closer to $\partial\Omega$ than $E$ (recall that $\diam E\asymp d(x_E,\partial \Omega)$).

Using this and \eqref{cadenas}, we obtain
\[
\begin{split}
\sum_{B\in E(W_M)}&\phi(r(B))\ell[\mathcal{C}(B^*)]^{p-1}\mu(B)\\
&\lesssim \sum_{B\in E(W_M)}\phi(r(B))\mu(B)\left(1+\log\frac{1}{\diam B}\right)^{p-1}\\
&\lesssim \sum_{B\in E(W_M)}\phi(r(B))\mu(B)\left(1+\log^p\frac{1}{\diam B}\right)\\
&\lesssim \sum_{B\in E(W_M)}\int_B\phi(r(B))\left(1+\log^p\frac{1}{d(y,\partial \Omega)}\right)d\mu(y)\\
&\lesssim \int_{B(\omega_E,\min\{1,c\diam E\})}\phi(r(E))\left(1+\log^p\frac{1}{d(y,\partial \Omega)}\right)d\mu(y).
\end{split}
\]
Since $X$ is a complete space, the boundary of the John domain $\Omega$ is porous in $S$, so we can apply Lemma \ref{porous} to $\omega_E$ with $r=\min\{1,c\diam E\}$ in order to obtain
\[
\begin{split}
\sum_{B\in E(W_M)}&\phi(r(B))\ell[\mathcal{C}(B^*)]^{p-1}\mu(B)\\
&\lesssim \phi(r(E))\mu[B(\omega_E,\min\{1,c\diam E\})]\left(1+\log^p\frac{1}{\min\{1,c\diam E\}}\right).
\end{split}
\]
Thus, we can check \eqref{suficp} for $\Omega$, obtaining
\[
\begin{split}
\sup_{E\in W_M}&\sum_{B\in E(W_M)}\frac{\phi(r(B))}{\phi(r(E))}\ell[\mathcal{C}(B^*)]^{p-1}r(E)^{(s-\gamma) p}\frac{\mu(B)}{\mu(E)}\\
&\lesssim\sup_{E\in W_M} \frac{\phi(r(E))}{\phi(r(E))}r(E)^{(s-\gamma) p}\left(1+\log^p\frac{1}{\diam E}\right)\frac{\mu[B(\omega_E,\min\{1,c\diam E\})]}{\mu(E)}\\
&\lesssim\sup_{E\in W_M} r(E)^{(s-\gamma) p}\left(1+\log^p\frac{1}{\diam E}\right)\left(\frac{\min\{1,2cr(E)\}}{r(E)}\right)^{n_\mu}\\
&<\infty,
\end{split}
\]
where we have used the doubling condition on $\mu$ and the fact that the ball $E$ is in the ball $B(\omega_E,\min\{1,c\diam E\})$. This last sum is finite as  $1+\log^p\frac{1}{t}\lesssim \frac{1}{t^{\eta p}}$, for  $0<\eta<s-\gamma$ if $t<1$.
\end{proof}
\begin{obs}
The growth condition on $\phi$ is not necessary. Moreover, the result also holds for a function $\phi$ satisfying $\phi(2t)\leq t^\delta \phi(t)$, where $\delta>\gamma-s$ (the case $\gamma=s$ is allowed here).

Also, it is interesting to compare this result with Theorem \ref{(p,p)} in the sense that here we just need $\mu$ to be a doubling measure in order to get the inequality, whereas in Theorem \ref{(p,p)} one asks $\mu$ to be a more regular measure.
\end{obs}

\section*{Acknowledgements}
The third author is supported by the Basque Government through the BERC 2018-2021 program and by Spanish Ministry of Science, Innovation and Universities through BCAM Severo Ochoa accreditation SEV-2017-0718. He is also supported by MINECO through the MTM2017-82160-C2-1-P project funded by (AEI/FEDER, UE), acronym ``HAQMEC'' and through ''la Caixa'' grant. The author also wants to thank IMAS (CONICET \& UBA) for its support for the author's stay at Buenos Aires in December 2017.

\printbibliography

\end{document}